\documentclass[11pt]{article}

\usepackage{amsthm}
\usepackage{amssymb}
\usepackage{amsmath,enumerate}
\usepackage{comment}
\usepackage{thm-restate}
\usepackage{url} 
\usepackage{hyperref}
\usepackage{graphicx}
\usepackage{subfigure}
\usepackage[noabbrev,capitalise]{cleveref}
\usepackage[affil-it]{authblk}
\usepackage{color}
\usepackage[normalem]{ulem}

\usepackage[margin=1in]{geometry}

\theoremstyle{plain}
\newtheorem{thm}{Theorem}[section]
\newtheorem{lem}[thm]{Lemma}

\newtheorem{conj}[thm]{Conjecture}

{\noindent \emph{Proof.} {}{#1}{}}{\hfill
	$\Diamond$\vspace{1em}}

\theoremstyle{plain} 
\newcommand{\thistheoremname}{}
\newtheorem{genericthm}[section]{\thistheoremname}

\theoremstyle{definition}

\def\less{\setminus}

\title{Some remarks on even-hole-free graphs}
\author{Zi-Xia Song\thanks{Supported by  NSF award DMS-1854903. E-mail address: {\tt Zixia.Song@ucf.edu}.}}
  \affil{ 
  { \small {Department  of Mathematics, University of Central Florida, Orlando, FL 32816, USA}}  
     }

\date{}
\begin{document}
\maketitle
\begin{abstract}
A vertex of a graph is {\it bisimplicial} if the set of its neighbors is the union of two cliques; a graph is {\it quasi-line} if   every vertex is bisimplicial.    A recent result of Chudnovsky and Seymour asserts that every non-empty even-hole-free graph has a bisimplicial vertex.   Both  Hadwiger's conjecture and the Erd\H{o}s-Lov\'asz Tihany conjecture have been shown to be true for quasi-line graphs,   but are open  for even-hole-free graphs. In this note, we prove that  for all $k\ge7$, every even-hole-free graph   with no $K_k$ minor is $(2k-5)$-colorable;   every even-hole-free graph $G$ with $\omega(G)<\chi(G)=s+t-1$ satisfies the  Erd\H{o}s-Lov\'asz Tihany conjecture  provided that $ t\ge s>  \chi(G)/3$. Furthermore, we   prove that every $9$-chromatic graph $G$ with $\omega(G)\le 8$ has a $K_4\cup K_6$ minor. Our proofs rely heavily on the structural result of Chudnovsky and Seymour  on even-hole-free graphs. 
\end{abstract}

\baselineskip 16pt

\section{Introduction}
  All graphs in this paper are finite and simple. For a graph $G$, we   use $V(G)$ to denote the vertex set, $E(G)$ the edge set, $|G|$ the number of vertices,   $e(G)$ the number of edges, $\delta(G)$  the minimum degree,  $\Delta(G)$  the maximum degree, $\alpha(G)$ the independence number, $\omega(G)$ the clique number and $\chi(G)$ the chromatic number. A graph $H$ is a minor of a graph $G$ if
$H$ can be obtained from a subgraph of $G$ by contracting edges. We write $G \succcurlyeq H$ if $H$ is
a minor of $G$. In those circumstances we also say that $G$ has an $H$ minor.   Our work is motivated by the  celebrated Hadwiger's conjecture~\cite{Had43}  and the   Erd\H{o}s-Lov\'asz Tihany conjecture~\cite{Erdos68}.

\begin{conj}[Hadwiger's conjecture~\cite{Had43}]\label{HC} For every integer $k \geq 1$, every graph with no $K_{k}$ minor is $(k-1)$-colorable. 
\end{conj}

\cref{HC} is  trivially true for $k\le3$, and reasonably easy for $k=4$, as shown independently by Hadwiger~\cite{Had43} and Dirac~\cite{Dirac52}. However, for $k\ge5$, Hadwiger's conjecture implies the Four Color Theorem~\cite{AH77,AHK77}.   Wagner~\cite{Wagner37} proved that the case $k=5$ of Hadwiger's conjecture is, in fact, equivalent to the Four Color Theorem, and the same was shown for $k=6$ by Robertson, Seymour and  Thomas~\cite{RST93}. Despite receiving considerable attention over the years, Hadwiger's conjecture remains wide open for $k\ge 7$ and is   widely considered among the most important problems in graph theory and has motivated numerous developments in graph coloring and graph minor theory. The best known upper bound on the chromatic number of graphs with no $K_k$ minor  is $O(k(\log \log k)^6)$ due to
Postle~\cite{Postle20}, improving a recent breakthrough of Norin, Postle, and the present author~\cite{NPS20} who improved a long-standing bound  obtained independently by Kostochka~\cite{Kostochka82,Kostochka84} and Thomason~\cite{Thomason84}.  We refer the reader to a recent survey by Seymour~\cite{Sey16} for further  background on \cref{HC}.\medskip

Throughout the paper, let  $s $ and $t $ be positive integers. A graph $G$ is \emph{$(s,t)$-splittable} if $V(G)$ can be partitioned into two sets $S$ and $T$ such that $\chi(G[S ]) \ge s$ and $\chi(G[T ]) \ge t$.  In 1968,   Erd\H{o}s~\cite{Erdos68}  published the following conjecture of Lov\'asz, which has since been known as the   Erd\H{o}s-Lov\'asz Tihany conjecture.

\begin{conj}[The Erd\H{o}s-Lov\'asz Tihany conjecture]\label{c:ELTC}
Let $G$ be a   graph with $\omega(G)<\chi(G)=s+t-1$, where   $t\ge s\ge2$  are integers.  Then $G$ is $(s,t)$-splittable.  
\end{conj}

 \cref{c:ELTC} is  hard, and few related results are known. 
The case $(2, 2)$ for  \cref{c:ELTC}  is trivial; the cases $(2, 3)$ and   $(3, 3)$  were shown by Brown and Jung~\cite{BJ69} in 1969;
Mozhan~\cite{Moz87} and Stiebitz~\cite{Sti87a} each independently showed the case $(2, 4)$ in 1987; 
the cases   $(3, 4)$ and  $(3, 5)$ were  settled by Stiebitz~\cite{Sti87b}  in 1988.  A relaxed version of \cref{c:ELTC} was proved in~\cite{Sti17}.\medskip

Recent work on both \cref{HC} and \cref{c:ELTC}  have also focused on proving the conjectures for certain classes of graphs. 
A vertex of a graph is  \emph{bisimplicial} if the set of its neighbors is the union of two cliques; a graph is \emph{quasi-line} if   every vertex is bisimplicial. Note that every line graph is  quasi-line and every quasi-line graph is claw-free~\cite{ChSe2012}.  A \emph{hole} in a graph   is an induced cycle of length at least four; a hole is \emph{even} if it has an even length. A graph is \emph{even-hole-free} if it contains no even hole. 
Hadwiger's conjecture has been shown to be true for     line graphs by Reed and Seymour~\cite{RS04}; quasi-line graphs by Chudnovsky and Ovetsky Fradkin~\cite{CF2008};   graphs  $G$ with $\alpha(G)\ge3$ and    no hole of length between $4$ and $2\alpha(G)-1$ by Thomas and the present author~\cite{ThomasSong}. 
Meanwhile,     the Erd\H{o}s-Lov\'asz Tihany conjecture   has also been verified to be true   for line graphs by Kostochka and Stiebitz~\cite{KS08}; 
quasi-line graphs,  and   graphs $G$ with $\alpha(G) = 2$ by Balogh, Kostochka, Prince and Stiebitz~\cite{BKPS09}; graphs  $G$ with $\alpha(G)\ge3$ and    no hole of length between $4$ and $2\alpha(G)-1$ by   the present author~\cite{Song19}.    \medskip

  Chudnovsky and Seymour~\cite{CS20}  recently  proved  a structural result on even-hole-free graphs. 

\begin{thm}[Chudnovsky and Seymour~\cite{CS20}]\label{t:evenholefree}  Let $G$ be a    non-empty even-hole-free graph. Then 
  $G$ has a bisimplicial vertex and
  $\chi(G)\le 2\omega(G)-1$. 
\end{thm}

It is unknown whether   \cref{HC} and \cref{c:ELTC} hold for even-hole-free graphs.  Using   \cref{t:evenholefree},  we prove in Section~\ref{s:Main} that  for all $k\ge7$, every even-hole-free graph   with no $K_k$ minor is $(2k-6)$-colorable;   every even-hole-free graph $G$ with $\omega(G)<\chi(G)=s+t-1$  satisfies \cref{c:ELTC}  provided that $t\ge s>  \chi(G)/3$.  It is worth noting that Kawarabayashi, Pedersen and Toft~\cite{KPT2011} observed that  if Hadwiger's conjecture holds, then the following conjecture   might be easier to settle than the   Erd\H{o}s-Lov\'asz Tihany conjecture. 

\begin{conj}[Kawarabayashi, Pedersen, Toft~\cite{KPT2011}]\label{c:eltcminor}
 Every graph $G$ satisfying $\omega(G)<\chi(G)=s+t-1$    has two vertex-disjoint subgraphs $G_1$ and $G_2$ such that $G_1\succcurlyeq K_s$ and $G_2 \succcurlyeq K_t$, where $  t\ge s\ge 2$ are integers. 
 \end{conj}

In the same paper \cite{KPT2011}, they settled   \cref{c:eltcminor} for   a few   additional values of $(s, t) \in \{ (2,6), (3, 6),   (4, 4),   (4, 5) \}$.  
  We end  Section~\ref{s:Main} by proving the $(4,6)$ case for  \cref{c:eltcminor}, that is,  
 we   prove that every   graph $G$ with $\chi(G)=9>\omega(G) $ has a $K_4\cup K_6$ minor.  Here $K_4\cup K_6$ denotes the disjoint union of  $K_4$ and $K_6$. \medskip

We need to introduce more notation.  Let $G$ be a graph. For a vertex $x\in V(G)$, we will use $N(x)$ to denote the set of vertices in $G$ which are adjacent to $x$.
We define $N[x] = N(x) \cup \{x\}$ and $d(x) = |N(x)|$.
If  $A, B\subseteq V(G)$ are disjoint, we say that $A$ is \emph{complete} to $B$ if each vertex in $A$ is adjacent to all vertices in $B$, and $A$ is \emph{anti-complete} to $B$ if no vertex in $A$ is adjacent to any vertex in $B$.
If $A=\{a\}$, we simply say $a$ is complete to $B$ or $a$ is anti-complete to $B$.  
The subgraph of $G$ induced by $A$, denoted $G[A]$, is the graph with vertex set $A$ and edge set $\{xy \in E(G) : x, y \in A\}$. We denote by $B \less A$ the set $B - A$,   and $G \less A$ the subgraph of $G$ induced on $V(G) \less A$, respectively.
If $A = \{a\}$, we simply write $B \less a$    and $G \less a$, respectively.  
 We say that  $G$ is \emph{$k$-chromatic} if $\chi(G)=k$.    An  \emph{$(s, t)$-graph}   is a connected $(s+t-1)$-chromatic graph which does not contain two vertex-disjoint subgraphs with chromatic number $s$ and $t$, respectively.   We use the convention   ``A :="  to mean that $A$ is defined to be
the right-hand side of the relation.\medskip

Finally, we shall make use of the following   results of Stiebitz~\cite{Sti87b, Sti96} and Mader~\cite{Mader68}.

  \begin{thm}[Stiebitz~\cite{Sti87b}]\label{t:comnbr} Suppose $G$ is  an $(s, t)$-graph with $t\ge s\ge2$. If   $\omega(G)\ge t$, then $\omega(G)\ge s+t-1$. 
 \end{thm}

\begin{thm}[Stiebitz~\cite{Sti96}]\label{t:mindegree}  Every graph $G$ satisfying $\delta(G)\ge s+t+1$  has two vertex-disjoint subgraphs $G_1$ and $G_2$ such that $\delta(G_1)\ge s \text{ and } \delta(G_2)\ge t$.
 \end{thm}

\begin{thm}[Mader~\cite{Mader68}]\label{t:mader}
 For every integer $p\le7$, every graph on   $n\ge p$ vertices  and  at least $ (p-2)n-{p-1\choose 2}+1$ edges has a $K_p$ minor.  
\end{thm}

\section{Main results}~\label{s:Main}
  Rolek and the present author~\cite[Theorem 5.2]{RolekSong2017} proved that if Mader's bound in \cref{t:mader} can be generalized to all values of $p$ (as in~\cite[Conjecture 5.1]{RolekSong2017}), then every graph with no $K_p$ minor is $(2p-6)$-colorable for all $p\ge7$. It is hard to prove~\cite[Conjecture 5.1]{RolekSong2017}.  We begin this section with an easy  result on coloring even-hole-free graphs with no $K_k$ minor, where $k\ge7$. It seems non-trivial to improve the bound in \cref{t:minor} to $2k-6$. 

\begin{thm}\label{t:minor}
For all $k\ge 7$, every even-hole-free graph with no $K_k$ minor is $(2k-5)$-colorable. 
\end{thm}

\begin{proof} Suppose the assertion is false. Let $G$ be an even-hole free graph with no $K_k$ minor  and  $\chi(G)\ge 2k-4$. We choose $G$ with $|G|$ minimum. Then $G$ is vertex-critical and $\chi(G)=2k-4$. Thus $\delta(G)\ge 2k-5$;  in addition, $G$ is connected and has no clique-cut. Suppose  $\omega(G)\ge k-1$. Let $K$ be a $(k-1)$-clique   in $G$. Then $G\less K$ is connected because $G$ has no clique-cut; by contracting $G\less K$ into a single vertex we obtain a $K_k$ minor, a contradiction.  Thus   $\omega(G)\le k-2$. 
Since $G$ is even-hole-free, by \cref{t:evenholefree},  $\chi(G)\le 2\omega(G)-1\le 2(k-2)-1=2k-5$, a contradiction.  \end{proof}

 We next prove  a   lemma which plays a key role  in the proof of \cref{t:ehfree} and \cref{t:case46}.

  \begin{lem}\label{l:main}
Let $G$ be a   graph  and $x\in V(G)$ with  $p:=\chi(G[N(x)])\ge2$. Let    $V_1, \ldots, V_p$ be the color classes of a proper $p$-coloring of $G[N(x)]$ with $|V_1|\ge\cdots\ge |V_p|\ge1$. If $   |V_r\cup\cdots\cup V_p|\le \chi(G)-r-1$ for some $r\in[p]$ with $2\le r\le p $, then $p\le \chi(G)-2$ and $G$ is $(r, \chi(G)+1-r)$-splittable. 
\end{lem}

\begin{proof} Let $G,  p, r, V_1, \ldots, V_p$ be as given in the statement. Note that  $p-r+1\le |V_r\cup\cdots\cup V_p|\le \chi(G)-r-1$ and so $p\le \chi(G)-2$ and $V(G)\less N[x]\ne\emptyset$.  Let $W:=V_1\cup\cdots\cup V_{r-1}$. Then  $\chi(G[\{x\}\cup W])=r$  
and $\chi(G\less W)\ge \chi(G)-(r-1)=\chi(G)+1-r$. It suffices to show that $ \chi(G\less(\{x\}\cup W))\ge \chi(G\less W)$.  Let $q:=\chi(G\less(\{x\}\cup W))\ge \chi(G\less W)-1\ge \chi(G)-r\ge2$ and let $U_1, \ldots, U_q$ be the color classes of a proper $q$-coloring of $G\less(\{x\}\cup W)$. Since $x$ is adjacent to $|V_r\cup\cdots\cup V_p|\le \chi(G)-r-1\le q-1$ vertices in $G\less W$, we see that $x$ is anti-complete to   $U_i $  for some $i\in[q]$. We may assume that $i=1$. Then  $U_1\cup\{x\}, U_2, \ldots, U_q$  form the color classes of a proper $q$-coloring of $G\less W$. Therefore, $ \chi(G\less(\{x\}\cup W))=q\ge \chi(G\less W)\ge \chi(G)-r+1$, as desired.  \end{proof}

We are now ready to prove that the Erd\H{o}s-Lov\'asz Tihany conjecture holds for even-hole-free graphs $G$ with $\omega(G)<\chi(G)=s+t-1$ if  $t\ge s>\chi(G)/3$. It would be nice if one can prove the same holds without the additional condition.

 \begin{thm}\label{t:ehfree}
Let $G$ be an  even-hole-free graph  with $\omega(G)<\chi(G)=s+t-1$, where $t\ge s\ge2$. If $s>\chi(G)/3$, then $G$  is $(s,t)$-splittable. 
\end{thm}

\begin{proof} Suppose the assertion is false. Let $G$ be a counterexample with $|G|$ minimum. Then   $G$ is vertex-critical; in addition, $G$ is an $(s,t)$-graph.   Thus $\delta(G)\ge \chi(G)-1=s+t-2$.    
 By \cref{t:comnbr}, $\omega(G)\le t-1$.   Since  $G$ is even-hole-free,  by \cref{t:evenholefree}, $G$ has a bisimplicial vertex  $v$ such that  
 $N(v)$ is the union of  two   cliques. Thus  $\alpha(G[N(v)])\le 2$, $\omega(G[N(v)])\le  t-2$ and 
 $$s+t-2=\chi(G)-1\le \delta(G)\le d(v)\le 2\omega(G[N(v)])\le 2t-4.$$
 It follows that  $t\ge s+2 \ge4$  and $\chi(G)=s+t-1\ge 2s+1$. 
 We next claim that $\Delta(G)\le |G|-2$. Suppose there exists $x\in V(G)$ such that $d(x)=|G|-1$. Then \[\chi(G\less x)=\chi(G)-1=s+(t-1)-1>\omega(G)-1=\omega(G\less x) \text{ and } t-1>s>\chi(G\less x)/3.\]  By the minimality of $G$,  $G\less x$ is $(s, t-1)$-splittable and thus $G$ is  $(s,t)$-splittable, a contradiction. Thus   $\Delta(G)\le |G|-2$, as claimed.  It follows that $V(G)\less N[v]\ne\emptyset$ and so $\chi(G[N[v]])\le \chi(G)-1$. Let $p:=\chi(N(v))$. Then  $p=\chi(G[N[v]])-1\le \chi(G)-2$.  Note that  
  $$p\ge \omega(G[N(v)])\ge d(v)/2\ge (\chi(G)-1)/2 \ge ((2s+1)-1)/2=  s\ge2.$$
Let $V_1, \ldots, V_p$ be the color classes of a proper $p$-coloring of $G[N(v)]$ with $2\ge |V_1|\ge \cdots\ge |V_p|\ge1$. Suppose $p\ge t-1$. Then $|V_{t-2}|=1$ because $d(v)\le 2t-4$. Therefore,   \[|V_t\cup\cdots\cup V_p|=p-t+1\le (\chi(G)-2)-t+1= \chi(G)-t-1.\]  By \cref{l:main} applied to $G$ and $v$ with $r=t$, we see that $G$ is $(s, t)$-splittable, a contradiction. Thus $s\le p\le t-2$.  Next, if $|V_s\cup\cdots\cup V_p|\le \chi(G)-s-1$, then  $G$ is $(s, t)$-splittable by applying \cref{l:main}   to $G$ and $v$ with $r=s$, a contradiction. Hence, 
 $|V_s\cup\cdots\cup V_p|\ge \chi(G)-s=t-1\ge 3$.  Note that $p-s+1\le (t-2)-2+1=t-3$,  and so   $|V_s|=2$ and 
  \[d(v)=(|V_1|+\cdots+|V_{s-1}|)+|V_s\cup\cdots\cup V_p|\ge  2(s-1)+t-1=2s+t-3.\]
 It follows that  
  $t-2\ge \omega(G[N(v)])\ge d(v)/2\ge (2s+t-3)/2$, which implies that  $t\ge 2s+1$. Thus   $\chi(G)=s+t-1\ge 3s$, contrary to the assumption that $3s>\chi(G)$.  
\end{proof}

Finally, we prove that   \cref{c:eltcminor} is true when  $(s,t)=(4,6)$. 

\begin{thm}\label{t:case46}
 Every   $9$-chromatic graph $G$ with $\omega(G)\le 8$ has a   $  K_4\cup K_6$ minor.  
\end{thm}
\begin{proof}
Suppose for a contradiction that  $G$ is a counterexample  to the statement with minimum number of vertices. Then $G$ is vertex-critical, and so $\delta(G)\ge 8$ and $G$ is connected. Suppose $G$   contains two vertex-disjoint subgraphs $G_1$ and $G_2$ such that $\chi(G_1)\ge 4$ and $ \chi(G_2)\ge 6$. Since Hadwiger's conjecture holds for  $k$-chromatic graphs with   $k\le6$, we see that     $  G_1\succcurlyeq K_4$ and $  G_2\succcurlyeq K_6$, a contradiction. Thus  $G$ is a $(4,6)$-graph, and so $\omega(G)\le5$ by  \cref{t:comnbr}.    Note that $G$ is not necessarily  contraction-critical, as a proper minor of $G$ may have clique number $9$. We  claim that  \medskip

\noindent {\bf Claim 1.}  \,\,  $2\le\alpha(G[N(x)])\le d(x)-7$ for   each  $x\in V(G)$.  

\begin{proof}  Let $x\in V(G)$. Since $\omega(G)\le5$ and $\delta(G)\ge8$, we see that $\alpha(G[N(x)])\ge2$.  Suppose   $\alpha(G[N(x)])\ge  d(x)-6$. Let $A$ be a maximum independent set of $G[N(x)]$. Let $G^*$ be obtained from $G$ by contracting $G[A\cup\{x\}]$ into a single vertex, say $w$. Note that $\omega(G^*)<8$ and   $G^*$ has no $K_4\cup K_6$ minor.  By the minimality of $G$,   $\chi(G^*)\le 8$.  Let $c: V(G^*)\to [8]$ be a proper $8$-coloring of $G^*$. Since $|N(x)\less A|=d(x)-|A|\le 6$, we may assume that $c(N(x)\less A)\subseteq [6]$ and $c(w)=7$. But then we obtain a proper $8$-coloring of $G$ from $c$ by coloring all the vertices in $A$ with color $7$ and the vertex $x$ with  color $8$, a contradiction. Thus  $2\le \alpha(G[N(x)])\le d(x)-7$, as claimed. \end{proof}

 By Claim 1,  $\delta(G)\ge9$.     Suppose $\delta(G)\ge 13$. By \cref{t:mindegree}, $G$   contains two vertex-disjoint subgraphs $G_1$ and $G_2$ such that $\delta(G_1)\ge 4 $ and $ \delta(G_2)\ge 8$. By 
 \cref{t:mader}, we see that $  G_1\succcurlyeq K_4$ and $  G_2\succcurlyeq K_6$, a contradiction. Thus $9\le \delta(G)\le 12$.   
 We next claim that \medskip

\noindent {\bf Claim 2.}\,\, $G[N(x)]$  is even-hole-free  and $\chi(G[N(x)])\le 2\omega(G[N(x)])-1$ for each $x\in V(G)$.  

\begin{proof} Let $x\in V(G)$. Suppose   $G[N(x)]$ contains an even hole $C$. Then  $\chi(G[V(C)\cup\{x\}]) =3$ and so $\chi(G\less (V(C)\cup\{x\}))\ge \chi(G) -3=6$. It is easy to see that  $G[V(C)\cup\{x\}]\succcurlyeq K_4$. Since Hadwiger's conjecture holds for $6$-chromatic graphs, we see that   $  G\less (V(C)\cup\{x\})$ has a  $K_6$ minor, and so $G$ has a  $K_4\cup K_6$ minor, a contradiction. Thus  $G[N(x)]$  is even-hole-free. By \cref{t:evenholefree}, $\chi(G[N(x)])\le 2\omega(G[N(x)])-1$. \end{proof}

Let $v\in V(G)$ with $d(v)= \delta(G)$, and let $p:=\chi(G[N(v)])$.  Since  $9\le d(v)\le 12$, we see that    $p\ge3$ by Claim 1. Suppose $G[N(v)]$ is $K_3$-free. By Claim 2, $p\le 2\omega(G[N(v)])-1=3$. Thus    $\chi(G[N[v]])=4$  and    $\chi(G\less N[v])= \chi(G\less N(v))\ge 9-3=6$, contrary to the fact that $G$ is a $(4,6)$-graph.  Thus $\omega(G[N(v)])\ge 3$.    Let $v_1, v_2, v_3\in N(v)$ be pairwise adjacent in $G$ and let $H:=G\less \{v,v_1, v_2, v_3 \}$. Then $G[\{v,v_1, v_2, v_3 \}]=K_4$ and \begin{align*}
2e(H)&\ge (d(v)-3)(|G\less N[v]|)+ (d(v)-4)\cdot |N(v)\less \{v_1, v_2, v_3\}|\\
&=(d(v)-3)(|G|-d(v)-1)+(d(v)-4)(d(v)-3)\\
&=(d(v)-3)(|H|-1).
\end{align*}
 Suppose $d(v)\in\{11,12\}$.  Then  $2e(H)\ge 8(|H|-1)$. By \cref{t:mader},    $H\succcurlyeq K_6$, and so $G$ has a  $K_4\cup K_6$ minor, a contradiction. This proves that  $ 9\le   d(v)\le 10$.    Then $p\ge4$ by Claim 1.  Since $\omega(G[N(v)])\le 4$, we see that $G[N(v)]$ has an anti-matching of size at least three. It follows that  $4\le p\le d(v)-3$. Let $V_1, \ldots, V_p$ be the color classes of a proper $p$-coloring of $G[N(v)]$ with $|V_1|\ge \cdots\ge |V_p|\ge1$. If  $p\in\{4,5\}$, then $|V_4|\le 2$ because $d(v)\le10$. Thus $|V_4\cup\cdots\cup V_p|\le 4=\chi(G)-4-1$. By \cref{l:main} applied to  $G$ and $v$ with $r=4$, we see that $G$ is $(4,6)$-splittable,  contrary to the fact that $G$ is a $(4,6)$-graph.  It remains to consider the case  $6\le p\le d(v)-3$. Since $d(v)\le10$, we see that  $|V_5|=1$.  Thus $|V_6\cup\cdots\cup V_p|=p-5\le (d(v)-3)- 5\le2=\chi(G)-6-1$. By \cref{l:main}    applied to  $G$ and $v$ with $r=6$, we see that $G$ is $(4,6)$-splittable, a contradiction. \medskip

This completes the proof of \cref{t:case46}.  
\end{proof}

\bibliographystyle{alpha}
\bibliography{Zixia}

\begin{thebibliography}{BKPS09}

\bibitem[AH77]{AH77}
K.~Appel and W.~Haken.
\newblock Every planar map is four colorable. {I}. {D}ischarging.
\newblock {\em Illinois J. Math.}, 21(3):429--490, 1977.

\bibitem[AHK77]{AHK77}
K.~Appel, W.~Haken, and J.~Koch.
\newblock Every planar map is four colorable. {II}. {R}educibility.
\newblock {\em Illinois J. Math.}, 21(3):491--567, 1977.

\bibitem[BJ69]{BJ69}
W.~G. Brown and H.~A. Jung.
\newblock On odd circuits in chromatic graphs.
\newblock {\em Acta Math. Acad. Sci. Hungar.}, 20:129--134, 1969.

\bibitem[BKPS09]{BKPS09}
J\'{o}zsef Balogh, Alexandr~V. Kostochka, Noah Prince, and Michael Stiebitz.
\newblock The {E}rd\H{o}s-{L}ov\'{a}sz {T}ihany conjecture for quasi-line
  graphs.
\newblock {\em Discrete Math.}, 309(12):3985--3991, 2009.

\bibitem[CF08]{CF2008}
Maria Chudnovsky and Alexandra~Ovetsky Fradkin.
\newblock Hadwiger's {C}onjecture for quasi-line graphs.
\newblock {\em Journal of Graph Theory}, 59(1):17--33, 2008.

\bibitem[CS12]{ChSe2012}
Maria Chudnovsky and Paul Seymour.
\newblock Claw-free graphs. {VII}. {Q}uasi-line graphs.
\newblock {\em J. Combin. Theory Ser. B}, 102(6):1267--1294, 2012.

\bibitem[CS19]{CS20}
Maria Chudnovsky and Paul Seymour.
\newblock Even-hole-free graphs still have bisimplicial vertices.
\newblock 2019.
\newblock arXiv:1909.10967.

\bibitem[Dir52]{Dirac52}
G.~A. Dirac.
\newblock A property of {$4$}-chromatic graphs and some remarks on critical
  graphs.
\newblock {\em J. London Math. Soc.}, 27:85--92, 1952.

\bibitem[Erd68]{Erdos68}
P.~Erd\H{o}s.
\newblock Problems.
\newblock In {\em Theory of {G}raphs ({P}roc. {C}olloq., {T}ihany, 1966)},
  pages 361--362. Academic Press, New York, 1968.

\bibitem[Had43]{Had43}
Hugo Hadwiger.
\newblock \"{U}ber eine {K}lassifikation der {S}treckenkomplexe.
\newblock {\em Vierteljahrsschrift der Naturforschenden Gesellschaft in
  Z\"{u}rich}, 88:133--142, 1943.

\bibitem[Kos82]{Kostochka82}
Alexandr~V. Kostochka.
\newblock The minimum {H}adwiger number for graphs with a given mean degree of
  vertices.
\newblock {\em Metody Diskret. Analiz.}, (38):37--58, 1982.

\bibitem[Kos84]{Kostochka84}
Alexandr~V. Kostochka.
\newblock Lower bound of the {H}adwiger number of graphs by their average
  degree.
\newblock {\em Combinatorica}, 4(4):307--316, 1984.

\bibitem[KPT11]{KPT2011}
Ken-ichi Kawarabayashi, Anders~Sune Pedersen, and Bjarne Toft.
\newblock The {E}rd\H{o}s-{L}ov\'{a}sz {T}ihany {C}onjecture and complete
  minors.
\newblock {\em J. Comb.}, 2(4):575--592, 2011.

\bibitem[KS08]{KS08}
Alexandr~V. Kostochka and Michael Stiebitz.
\newblock Partitions and edge colourings of multigraphs.
\newblock {\em Electron. J. Combin.}, 15(1):Note 25, 4, 2008.

\bibitem[Mad68]{Mader68}
W.~Mader.
\newblock Homomorphies\"atze f\"ur {G}raphen.
\newblock {\em Math. Ann.}, 178:154--168, 1968.

\bibitem[Moz87]{Moz87}
N.~N. Mozhan.
\newblock Twice critical graphs with chromatic number five.
\newblock {\em Metody Diskret. Analiz.}, (46):50--59, 73, 1987.

\bibitem[NPS20]{NPS20}
Sergey Norin, Luke Postle, and Zi-Xia Song.
\newblock Breaking the degeneracy barrier for coloring graphs with no {$K_t$}
  minor.
\newblock 2020.
\newblock arXiv:1910.09378v2.

\bibitem[Pos20]{Postle20}
Luke Postle.
\newblock An even better density increment theorem and its application to
  hadwiger?s conjecture.
\newblock 2020.
\newblock arXiv:2006.14945.

\bibitem[RS04]{RS04}
Bruce Reed and Paul Seymour.
\newblock Hadwiger's conjecture for line graphs.
\newblock {\em European J. Combin.}, 25(6):873--876, 2004.

\bibitem[RS17]{RolekSong2017}
Martin Rolek and Zi-Xia Song.
\newblock Coloring graphs with forbidden minors.
\newblock {\em Journal of Combinatorial Theory, Series B}, 127:14--31, 2017.

\bibitem[RST93]{RST93}
Neil Robertson, Paul Seymour, and Robin Thomas.
\newblock Hadwiger's conjecture for {$K_6$}-free graphs.
\newblock {\em Combinatorica}, 13(3):279--361, 1993.

\bibitem[Sey16]{Sey16}
Paul Seymour.
\newblock Hadwiger's conjecture.
\newblock In {\em Open problems in mathematics}, pages 417--437. Springer,
  2016.

\bibitem[Son19]{Song19}
Zi-Xia Song.
\newblock Erd\H{o}s-{L}ov\'{a}sz {T}ihany conjecture for graphs with forbidden
  holes.
\newblock {\em Discrete Math.}, 342(9):2632--2635, 2019.

\bibitem[Sti87]{Sti87a}
Michael Stiebitz.
\newblock {$K_5$} is the only double-critical {$5$}-chromatic graph.
\newblock {\em Discrete Math.}, 64(1):91--93, 1987.

\bibitem[Sti88]{Sti87b}
Michael Stiebitz.
\newblock On {$k$}-critical {$n$}-chromatic graphs.
\newblock In {\em Combinatorics ({E}ger, 1987)}, volume~52 of {\em Colloq.
  Math. Soc. J\'{a}nos Bolyai}, pages 509--514. North-Holland, Amsterdam, 1988.

\bibitem[Sti96]{Sti96}
Michael Stiebitz.
\newblock Decomposing graphs under degree constraints.
\newblock {\em J. Graph Theory}, 23(3):321--324, 1996.

\bibitem[Sti17]{Sti17}
Michael Stiebitz.
\newblock A relaxed version of the {E}rd\H{o}s-{L}ov\'{a}sz {T}ihany
  conjecture.
\newblock {\em J. Graph Theory}, 85(1):278--287, 2017.

\bibitem[Tho84]{Thomason84}
Andrew Thomason.
\newblock An extremal function for contractions of graphs.
\newblock {\em Math. Proc. Cambridge Philos. Soc.}, 95(2):261--265, 1984.

\bibitem[TS17]{ThomasSong}
Brian Thomas and Zi-Xia Song.
\newblock Hadwiger's conjecture for graphs with forbidden holes.
\newblock {\em SIAM Journal of Discrete Mathematics}, 31:1572--1580, 2017.

\bibitem[Wag37]{Wagner37}
K.~Wagner.
\newblock \"{U}ber eine {E}igenschaft der ebenen {K}omplexe.
\newblock {\em Mathematische Annalen}, 114:570--590, 1937.

\end{thebibliography}

\end{document}